\documentclass{amsart}
\usepackage{amssymb,amsmath,latexsym}
\usepackage{amsthm}
\usepackage{fontenc}
\usepackage{amssymb}
\numberwithin{equation}{section}

\newtheorem{theorem}{Theorem}[section]

\newtheorem{proposition}[theorem]{Proposition}
\newtheorem{lemma}[theorem]{Lemma}

\begin{document}
\author{Alexander E Patkowski}
\title{A note on applications of the extension of Abel's lemma}

\maketitle
\begin{abstract} We offer some new applications of an extension of Abel's lemma, as well as its more general form established
by Andrews and Freitas. A nice connection is established between this lemma and series involving the Riemann zeta function. \end{abstract}

% AMS keywords (used in AMS journals)

% AMS subject classifications (used in AMS journals)

\section{Introduction}
In a paper by Andrews and Freitas [4], the extension of Abel's lemma was further generalized and several new $q$-series were established. Recall that Abel's lemma is the simple result that $\lim_{z\rightarrow 1^{-}}(1-z)\sum_{n\ge0}a_nz^n =\lim_{n\rightarrow\infty}a_n.$ We use the shifted factorial notation $(a)_n=a(a+1)\cdots (a+n-1)$ in this paper [2]. Their result may
be stated as follows.

\begin{proposition} ([4, Proposition 1.2]) Let $f(z)=\sum_{n\ge0} \alpha_n z^n$ be analytic for $|z|<1,$ and assume that for some positive integer $M$ and a fixed
complex number $\alpha$ we have that (i) $\sum_{n\ge0}(n+1)_{M}(\alpha_{M+n}-\alpha_{M+n-1})$ converges,  and (ii) $\lim_{n\rightarrow \infty} (n+1)_{M}(\alpha_{M+n}-\alpha)=0.$  Then 
$$\frac{1}{M}\lim_{z\rightarrow1^{-}}\left(\frac{\partial^M}{\partial z^M}(1-z)f(z)\right)=\sum_{n\ge0}(n+1)_{M-1}(\alpha-\alpha_{n+M-1}).$$
\end{proposition}
The formula being generalized here is given in [3, Proposition 2.1], where it was used to find generating functions for special values for certain $L$-functions. A corollary 
of the extension of Abel's lemma was also given in [7].
\par 
In the work [1] we find a simple formula attributed there to Christian Goldbach,
\begin{equation}\sum_{n\ge0}(1-\zeta(n+2))=-1.\end{equation}
Now it does not appear any connection has been made between the extension of Abel's lemma and this result, but as we shall demonstrate, it is a simple consequence of it. 
To this end, we shall prove some more general formulas in the next section which we believe are interesting applications of the Andrews-Freitas formula. For this, we will use a result from the work [6]. For some relevant series identities of a similar nature see also [5,8]. The main theorems presented here appear to differ considerably from previous similar examples, such as [8, pg.24, eq.(2.4)] where sums involving $(n)_M$ run over $M,$ since ours run over $n.$ 
\section{Some new theorems}
This section establishes some interesting theorems, which we hope will add value to the Andrews-Freitas formula. For convenience in our proofs, we decided to write down a simple
lemma.
\begin{lemma} If $f(z)$ has no factor $(1-z)^{-1},$ then we may write
$$\lim_{z\rightarrow1^{-}}\frac{\partial^M}{\partial z^M}(1-z)f(z)=-Mf^{(M-1)}(1).$$
\end{lemma}
\begin{proof} Put $f_1(z)=(1-z),$ and $f_2(z)=f(z).$ Then by the Leibniz rule,
$$\lim_{z\rightarrow1^{-}}\frac{\partial^M}{\partial z^M}(1-z)f(z)=\lim_{z\rightarrow1^{-}}\sum_{j\ge0}\binom {M}{j} f_1^{(j)}f_2^{(M-j)}$$
$$=\lim_{z\rightarrow1^{-}}\binom {M}{1}f_1^{(1)}f_2^{(M-1)}$$
$$=-M\lim_{z\rightarrow1^{-}}f_2^{(M-1)}(z),$$
because if $j=0$ then the term in the sum, $f_1^{(0)},$ is 0 when $z\rightarrow1^{-},$ and for $j>1,$ $f_1^{(j)}=0.$
\end{proof}
As usual, we denote $\gamma$ to be Euler's constant [2]. We also define the polygamma function [2] to be 
the $(M+1)$-th derivative of the logarithm of the Gamma function: $\psi^{(M)}(z)=\frac{\partial^{M+1}}{\partial z^{M+1}}(\log\Gamma(z)).$

\begin{theorem} For positive integers $M,$ we have that 
$$\sum_{n\ge0}(n+1)_{M-1}(1-\zeta(n+M+1))$$
$$=(-1)^{M-1}\sum_{j\ge0}\binom {M-1}{j} j!\psi^{(M-j-1)}(1)+(M-1)!(-1)^M+\gamma(-1)^{M-1}(M-1)!.$$
\end{theorem}
\begin{proof} First we write down the well-known Taylor expansion of the digamma function [1, 2], for $|z|<1,$
\begin{equation} \psi^{(0)}(z+1)=-\gamma-\sum_{k\ge1}\zeta(k+1)(-z)^k.\end{equation}
It is a trivial exercise to re-write (2.1) as 
\begin{equation}-z^{-1} \psi^{(0)}(1-z)-z^{-1}\gamma=\sum_{k\ge0}\zeta(k+2)z^k. \end{equation}
Inserting the functional equation for $\psi^{(0)}(z),$ given by [1, 2]
\begin{equation} \psi^{(0)}(z+1)=\psi^{(0)}(z)+\frac{1}{z},\end{equation}
into (2.2) and multiplying by $(1-z)$ gives
\begin{equation} -z^{-1}(1-z)(\psi^{(0)}(2-z)-(1-z)^{-1})-z^{-1}(1-z)\gamma=(1-z)\sum_{k\ge0}\zeta(k+2)z^k,\end{equation}
Now applying Proposition 1.1 with $\alpha_n=\zeta(n+2),$ and involving (2.4) gives the theorem after applying Lemma 2.1.
\end{proof}

For $M=1$ Theorem 2.2 specializes to Goldbach's formula (1.1). We shall denote $S(n,l)$ to be the Stirling numbers of the second kind [2]. 
\begin{theorem} For positive integers $M$ and $N,$ we have that 
$$\sum_{n\ge0}(n+1)_{M-1}(n+M+1)^{N}(1-\zeta(n+M+1))=\sum_{l\ge1}^{N}S(N+1,l+1)(-1)^{l+1}g_{M,l}$$
$$+(-1)^{M-1}\sum_{j\ge0}\binom {M-1}{j} j!\psi^{(M-j-1)}(1)+\gamma(-1)^{M-1}(M-1)!,$$
where for $l\ge0,$
$$g_{M,l}:=-\sum_{j\ge0}\binom {M-1}{j}(-1)^{M-1-j}\psi^{(l+M-1-j)}(1)\frac{(l-1)!}{(l-1-j)!}.$$
\end{theorem}
\begin{proof} From [6, Corollary 2], we find the delightful formula for integers $N\ge1$ and $\Re(a)>0,$
\begin{equation}\sum_{k\ge2}k^{N}z^{k}\zeta(k,a)=\sum_{l\ge1}^{N}S(N+1,l+1)l!\zeta(l+1,a-z)z^{l+1}-z(\psi^{(0)}(a-z)-\psi^{(0)}(a)),\end{equation}
for $|z|<|a|.$ $\zeta(s,a)$ is the Hurwitz zeta function [2]. We have also corrected the stated formula by instead having $N\ge1.$ We have also shifted the sum by replacing $l$ by $l+1$ for our convenience. Now $\lim_{n\rightarrow\infty}\zeta(n,a)=0$
if $a>1,$ $1$ if $a=1,$ $+\infty$ if $0<a<1.$ Hence the formula (2.5) is of the type of interest to our study only if $a=1.$ So, in that case, we put $a=1,$ and re-write (2.5) as 
\begin{equation}\sum_{k\ge2}k^{N}z^{k}\zeta(k)=\sum_{l\ge1}^{N}S(N+1,l+1)l!\zeta(l+1,1-z)z^{l+1}-z(\psi^{(0)}(1-z)-\psi^{(0)}(1)).\end{equation}
Differentiating (2.3) $l$ times we get that 
\begin{equation} \psi^{(l)}(2-z)=\psi^{(l)}(1-z)+(1-z)^{-l-1}(-1)^l l!.\end{equation}
Now using equation [1, eq.(2.15)], we have
\begin{equation}\sum_{k\ge0}k^Nz^k=\sum_{k\ge1}k^Nz^k=\sum_{l\ge0}^{N}S(N+1, l+1)l!(1-z)^{-l-1}z^{l+1}.\end{equation}
Now $S(n,1)=1$ for all non-negative integers $n,$ so we may write (2.8) for $N\ge1$ as
\begin{equation}\sum_{k\ge0}k^Nz^k=z(1-z)^{-1}+\sum_{l\ge1}^{N}S(N+1, l+1)l!(1-z)^{-l-1}z^{l+1}.\end{equation}
Using $\psi^{(l)}(z)=(-1)^{l+1}l!\zeta(l+1,z),$ and (2.7), we re-write (2.6) as 
\begin{equation}\sum_{k\ge2}k^{N}z^{k}\zeta(k)=\sum_{l\ge1}^{N}S(N+1,l+1)((-1)^{l+1}\psi^{(l)}(2-z)+(1-z)^{-l-1}l!)z^{l+1}\end{equation}
$$-z(\psi^{(0)}(1-z)-\psi^{(0)}(1)).$$
Now comparing equation (2.9) with (2.10), and noting $\psi^{(0)}(1)=-\gamma,$ we see that we have that
\begin{equation}\sum_{k\ge2}k^{N}z^{k}(\zeta(k)-1)=\sum_{l\ge1}^{N}S(N+1,l+1)(-1)^{l+1}\psi^{(l)}(2-z)z^{l+1}+z-z(1-z)^{-1}\end{equation}
$$-z(\psi^{(0)}(1-z)+\gamma).$$
Now we choose $\alpha_n=(n+2)^{N}(\zeta(n+2)-1)$ and note that since $1$ is removed from the first term in $\zeta(s)$ that $\lim_{n\rightarrow\infty}(n+2)^{N}(\zeta(n+2)-1)=0,$ since exponential growth is faster than polynomial growth. The far right side of equation (2.11) may be construed as (2.2). Multiplying both sides by $z^{-2},$ and applying Proposition 1.1 we use the formula 
\begin{equation}\lim_{z\rightarrow1^{-}}\frac{\partial^M}{\partial z^M}((1-z)\psi^{(l)}(2-z)z^{l-1})\end{equation}
$$=-M\lim_{z\rightarrow1^{-}}\frac{\partial^{M-1}}{\partial z^{M-1}}(\psi^{(l)}(2-z)z^{(l-1)})$$
$$=-M\sum_{j\ge0}\binom {M-1}{j}(-1)^{M-1-j}\psi^{(l+M-1-j)}(1)\frac{(l-1)!}{(l-1-j)!}.$$
We employed the trivial formula $\lim_{z\rightarrow1^{-}}\frac{\partial^M}{\partial z^M}(z^{l})=l!/(l-M)!$ in the last line. This proves the theorem after noting that the $M$-th derivative of $(1-z)z^{-1}-z^{-1}=-1$ is $0.$
\end{proof}
Note that since $N\ge1,$ Theorem 2.3 is not a generalization of Theorem 2.2 and so Theorem 2.2 is not redundant. Further, for integers $N\ge1,$ we have that $\psi^{(N)}(1)=(-1)^{N+1}N!\zeta(N+1).$
\section{Conclusion}
The conclusion we have come to here is that the summation formula that was established to prove interesting $q$-series identities may also be used to prove 
identities for series involving the Riemann zeta function. Some further interest should be directed toward finding expressions for sums of the form
$$\sum_{n\ge0}a_n(L(n+\sigma+1)-1),$$
where the $a_n$ are appropriately chosen for the series to converge, and $L(s)$ is a Dirichlet series which is assumed to have its first term to be $1$ and converges when $\Re(s)>\sigma.$ 
We believe this is a curious incidence where attractive results in one area of mathematics may be grouped
as a consequence of a formula which has produced attractive results in another area. 

1390 Bumps River Rd. \\*
Centerville, MA
02632 \\*
USA \\*
E-mail: alexpatk@hotmail.com
\end{document}